\documentclass{amsart}
\usepackage{geometry}
\usepackage{lmodern}
\usepackage{thmtools}
\usepackage{amssymb,amsmath,amsthm}
\usepackage{hyperref}
\usepackage[capitalize]{cleveref}
\usepackage{setspace}
\usepackage{enumitem}
\usepackage[
            backend=biber,
            style=alphabetic,
            url=false,
            maxnames=10,
            minnames=5
]{biblatex}

\title{Structural complexity vs.\ computational complexity}
\author{Johanna N.Y.\ Franklin}
\author{Dino Rossegger}
\author{Dan Turetsky}
\newcommand{\G}{\mathcal{G}}

\renewcommand{\phi}{\varphi}

\mathchardef\mhyph="2D 
\newcommand{\A}{\mathcal{A}}
\newcommand{\C}{\mathcal{C}}
\newcommand{\B}{\mathcal{B}}

\newcommand{\ol}[1]{\overline{#1}}

\newcommand{\mc}[1]{\mathcal{#1}}
\newcommand{\ra}{\rightarrow}
\newcommand{\Ra}{\Rightarrow}

\newcommand{\Pinf}[1]{\Pi^{\mathrm{in}}_{#1}}
\newcommand{\Picom}[1]{\Pi^{\mathrm{c}}_{#1}}
\newcommand{\Sinf}[1]{\Sigma^{\mathrm{in}}_{#1}}
\newcommand{\Sicom}[1]{\Sigma^{\mathrm{c}}_{#1}}
\newcommand{\dSinf}[1]{d\mhyph\Sigma^{\mathrm{in}}_{#1}}

\newcommand{\ba}{\bar a}

\newcommand{\ock}{\omega_{1}^{\mathrm{ck}}}

\newcommand{\nat}{\omega}

\DeclareMathOperator{\restrict}{\upharpoonright}
\DeclareMathOperator{\concat}{^\smallfrown}

\newcommand{\define}[1]{\textbf{#1}}

\declaretheorem[name={Theorem}]{theorem}

\declaretheorem[name={Lemma}, sibling=theorem]{lemma}

\declaretheorem[name={Claim}, numberwithin=theorem]{claim}
\declaretheorem[name={Definition}, style=definition]{definition}
\declaretheorem[name={Remark},numbered=no, style=definition]{remark}

\declaretheorem[name={Question}]{question}

\Crefname{question}{Question}{Questions}

\address{Department of Mathematics, Hofstra University} \email{johanna.n.franklin@hofstra.edu}
\address{Institute of Discrete Mathematics and Geometry, Technische Universit\"at Wien} \email{dino.rossegger@tuwien.ac.at}
\address{School of Mathematics and Statistics, Victoria University of Wellington} \email{dan.turetsky@vuw.ac.nz}
\thanks{The authors would like to thank the Erwin Schrödinger International Institute for Mathematics and Physics (ESI) in Vienna for its hospitality and support during the thematic program `Reverse Mathematics'. The first author was supported in part by Simons Foundation Collaboration Grant \#420806, which also provided travel support to the second author. The work of the second author was supported by the Austrian Science Fund (FWF) 10.55776/PIN1878224 and 10.55776/PAT4699225. The third author was supported by the Royal Society of New Zealand through the Marsden grant 23-VUW-118.} 
\begin{document} 

\begin{abstract}
We consider highness in the context of computable structure theory and, particularly, the Scott rank of a structure. We define highness for Scott rank and highness for computably defined Scott rank $\alpha$ and characterize them in terms of the ability to compute $\Delta^0_\beta$ sets for appropriate $\beta$. We close with a discussion of the index sets of structures with a given Scott rank or computably defined Scott rank and a few words about highness for noncomputable Scott ranks.
\end{abstract}

\maketitle 

Highness allows us to describe an oracle as maximally powerful in some context. The notion was first introduced by Soare in \cite{s72} in the context of degree theory; here, the context was that of the Turing jump and the $\Delta^0_2$ degrees. Since then, highness has been considered in the context of randomness \cite{fsy-bases} and computable structure theory \cite{cft23,cft}. In this paper, we study highness not in the general context of computable structure theory but in that of structures of particular Scott ranks.

We refer the reader to the text by Ash and Knight~\cite{ash2000} for background on both effective and noneffective formula complexity.

We will use the following definition of Scott rank due to Montalbán~\cite{montalban2015a}.
\begin{definition}
  A structure $\A$ has \define{Scott rank $\alpha$} for $\alpha<\omega_1$ if $\alpha$ is least so that the automorphism orbit of every tuple in $\A$ is $\Sinf{\alpha}$-definable.
  It has \define{computably defined Scott rank $\alpha$} if $\alpha$ is least so that every automorphism orbit is $\Sicom\beta$-definable for $\beta\leq \alpha$.
\end{definition}
A computable structure $\A$ has \define{computable Scott rank} if it has Scott rank $\alpha<\ock$. 
\begin{remark}
Suppose $\A$ has Scott rank $\lambda$, where $\lambda$ is a limit ordinal. If $\bar a \in \A$, then by definition $\bar a$'s orbit is definable by a $\Sinf{\lambda}$ formula.  But a $\Sinf{\lambda}$ formula is a disjunction of formulas which are each $\Sinf{\beta}$ for some $\beta < \lambda$.  So by passing to the disjunct satisfied by $\bar a$, we may assume that the formula is $\Sinf{\beta}$ for some $\beta < \lambda$.
The same reasoning applies for computably defined Scott rank $\lambda$.
\end{remark}
Note that by this remark there are structures of computably defined Scott rank $\ock$ even though there are no $\Sicom\ock$-formulas. In fact, we will see (\cref{thm:limitiffcompdefined}) that every computable structure with Scott rank $\ock$ has computably defined Scott rank $\ock$.

Now we can define highness in this context. Calvert, Franklin, and Turetsky defined a set to be \define{high for isomorphism} if it can compute an isomorphism between any two isomorphic computably presented structures \cite{cft23}. In this case, we define highness for structures of a particular Scott rank.

\begin{definition} Let $\mathbf d$ be a Turing degree and $\alpha<\omega_1$.
  \begin{enumerate}
    \item $\mathbf d$ is \define{high for Scott rank $\alpha$} if for every computable structure $\A$ of Scott rank $\alpha$ and for all computable $\B\cong \A$, $\mathbf d$ computes an isomorphism $\B\to \A$.
    \item $\*d$ is \define{high for computably defined Scott rank $\alpha$} if for every structure $\A$ of computably defined Scott rank $\alpha$ and for all computable $\B\cong \A$, $\*d$ computes an isomorphism $\B\to\A$.
  \end{enumerate}
\end{definition}

Suppose $\A$ and $\B$ are computable structures of Scott rank at most $\alpha$, and $\C$ is any computable structure of Scott rank $\alpha$.  Then if $\A\oplus\C$ is the disjoint union, correctly defined, it has Scott rank exactly $\alpha$.  Further, any isomorphism $\B\oplus\C \to \A\oplus C$ gives an isomorphism $\B\to \A$.  Thus we could have equivalently defined highness for Scott rank $\alpha$ in terms of structures of Scott rank at most $\alpha$, rather than in terms of structures of Scott rank exactly $\alpha$.

It follows that for $\alpha \le \beta$, highness for Scott rank $\beta$ implies highness for Scott rank $\alpha$. The same holds for highness for computably defined Scott rank.

The Scott rank of a structure is a structural measure of complexity, giving upper bounds on the quantifier complexity of definable relations on the structure. Meanwhile, the least Turing degree computing isomorphisms between any two computable copies of a structure---commonly referred to as its \define{degree of categoricity}~\cite{fokina2010}---gives a computational bound: If two computable structures $\+A$ and $\+B$ are computably isomorphic and $R\subseteq\nat$ is a definable relation, then the Turing degrees of $R^\+A$ and $R^\+B$ are the same, and we can say that $\+A$ and $\+B$ possess the same computational properties. However, it is well known that this does not hold if $\+A$ and $\+B$ are not computably isomorphic. On the other hand, if the degree of categoricity of $\+A$ is $\*d$, then modulo $\*d$, $\+A$ and $\+B$ are computationally indistinguishable and hence degrees of categoricity provide a computational complexity measure.
Under these considerations, the immediate corollary of our main result that any $\Delta^0_{2\alpha+1}$-complete set is high for Scott rank $\alpha$ shows that structural simplicity limits computational complexity: No structure of Scott rank $\alpha$ can have degree of categoricity greater than $\mathbf 0^{(2\alpha+1)}$ (or, for finite $\alpha$, $\mathbf 0^{(2\alpha)}$).

The paper is organized as follows. In Section \ref{sec:cSr}, we discuss highness for computable Scott ranks and prove that for $0 < \alpha < \ock$, $\*d$ is high for Scott rank $\alpha$ if and only if $\*d$ computes every $\Delta^0_{2\alpha+1}$ set (Theorem \ref{thm:highforsr}).  In Section \ref{sec:cdSr}, we characterize highness for computably defined Scott rank $\alpha$ as computing every $\Delta^0_{\alpha+1}$ set for limit $\alpha$ and every $\Delta^0_{\alpha+2}$ set for successor $\alpha$ (Theorem \ref{thm:highforcompSR}). In Section \ref{sec:index}, we turn our attention to index sets and prove that the index set of structures with Scott rank $\alpha$ and the index set of structures with computably defined Scott rank $\alpha$ are both $\Sigma^0_{2\alpha+2}$-complete for all $\alpha< \ock$ (Theorem \ref{thm:indexsets}). Finally, in Section \ref{sec:ncSr}, we discuss highness for Scott ranks $\ock$ and $\ock+1$.

\section{Highness for computable Scott ranks}\label{sec:cSr}
The main goal of this section is to prove the following theorem.
\begin{theorem}\label{thm:highforsr}
For $0 < \alpha < \ock$, $\*d$ is high for Scott rank $\alpha$ if and only if $\*d$ computes every $\Delta^0_{2\alpha+1}$ set.\footnote{We are stating the theorem in this fashion to unify the finite and infinite cases.  If we were to write the criteria in terms of computing jumps, we would need to write $0^{(2n)}$ for $n < \omega$ and $0^{(2\alpha+1)}$ for $\alpha \ge \omega$.}
\end{theorem}

The easy direction is taken care of by the following lemma.
\begin{lemma}
  For $0<\alpha<\ock$, if $\*d$ computes every $\Delta^0_{2\alpha+1}$ set, then $\*d$ is high for Scott rank $\alpha$.
\end{lemma}
\begin{proof}
  Recall from Karp~\cite{karp1965b} that $(\A,\bar a)\leq_\alpha (\B,\bar b)$ if and only if every $\Sinf{\alpha}$ formula true of $(\B,\bar b)$ is true of $(\A,\bar a)$. Fix a computable structure $\A$ of Scott rank $\alpha$ and an isomorphic copy $\B$. Then 
  \[ B=\{ (\bar a,\bar b) \colon (\A, \bar a)\leq_\alpha (\B, \bar b)\}\]
  is a back-and-forth set for $\A$ and $\B$. By a straightforward induction on the back-and-forth relations $\leq_\alpha$ we have that $B$ is $\Delta_{2\alpha+1}^0$. Let us quickly illustrate the limit step. For $\alpha$ limit, we have that $(\A, \bar a)\leq_\alpha (\B, \bar b)$ if and only if $(\forall \beta<\alpha) (\A,\bar a)\leq_\beta (\B,\bar b)$. This is $\Pi^0_{\alpha}$, so any set computing all $\Delta_{2\alpha+1}^0=\Delta_{\alpha+1}^0$ sets can compute it. Furthermore, given $B$, we can compute an isomorphism between $\A$ and $\B$ using a standard back-and-forth argument.
\end{proof}
\subsection{Base case: Highness for Scott rank 1}
We now lay the groundwork for the base case of \cref{thm:highforsr}.
See~\cite{downey2011} for an overview of limitwise monotonic functions and their implications, including proofs of the unproven lemmas here. 

\begin{definition}\label{def:alimmon}
  A function $F$ is \define{$\mathbf{a}$-limitwise monotonic} if there is an $\mathbf a$-computable approximation function $f$ such that  for all $x$,
  \begin{enumerate}
    \item $F(x)=\lim_s f(x,s)$,
    \item and for all s, $f(x,s)\leq f(x,s+1)$.
  \end{enumerate}
\end{definition}
\begin{lemma}[Harris~\cite{harris2008} and Kach~\cite{kach2008}]\label{prop:liminfiff0'limmon}
   A function $F$ is $\mathbf 0'$-limitwise monotonic if and only if there is a computable function $g$ such that $F(n)=\liminf_s g(n,s)$. 
\end{lemma}
\begin{lemma}\label{lem:0''liminfmod}
  The set $\emptyset''$ has a modulus $f$ that is liminf computable, i.e., $f(n)=\liminf_s g(n,s)$ for some computable function $g$.
 \end{lemma}
 \begin{proof} Let $m$ be the modulus of computation for $\emptyset'$, i.e., 
  \[ m(n)=\mu s [\emptyset'_s\restrict n=\emptyset'\restrict n].\]
  Then it is straightforward to see that $m$ is limitwise monotonic and, relativizing this, we can define the analogous function $m_1$ for $\emptyset''$, which is $\mathbf 0'$-limitwise monotonic. The function $f$ to be defined as $f(n)=m(n)+m_1(n)$ has the property that if $g$ dominates $f$, then $g\geq_T \emptyset '$ as $m$ is a modulus but also $g\geq_T \emptyset''$, since $g$ dominates $m_1$ and any function able to compute $\emptyset'$ can run the $\emptyset'$-computable approximation of $\emptyset''$ for which $m_1$ is a modulus. Thus, $f$ is a modulus for $\emptyset''$. It is clearly $\mathbf 0'$-limitwise monotonic as limitwise monotonicity is closed under addition. By \cref{prop:liminfiff0'limmon} it is liminf computable.
 \end{proof}
 The following lemma provides the base case of \cref{thm:highforsr}. In an unpublished preprint Csima and Stephenson have proven the same result using a different construction.  
 \begin{lemma}\label{lem:sr1dgcat''}
   There is a computable structure $\A$ with Scott rank $1$ and degree of categoricity $\mathbf 0''$. In particular $\A$ has two computable copies $\mathcal{G}$ and $\B$ so that every isomorphism between $\+G$ and $\B$ computes every $\Delta^0_3$ set.
 \end{lemma}
 \begin{proof}
   We will build two copies of $\A$, $\mc G$ and $\B$ so that $\mathbf 0''$ is the least degree computing an isomorphism between them. Fix a liminf computable modulus for $\mathbf 0''$ given via the computable function $g$. Our structure will be a graph with additional unary relation symbols $R_n$ that partition the universe into infinitely many infinite sets. We will call the part of the universe where $R_n$ holds the \define{$n$-slice}.

   Let us start by describing the construction of the $n$-slice in $\B$ for fixed $n$ in stages.
   At stage $0$, the slice is empty. We now describe the construction at stage $s$, having finished our construction at stage $s-1$.
   \begin{enumerate}
     \item Add a new point $b_s$ to the slice.
     \item If $g(n,s)\geq g(n,s-1)$, add a cycle of length $s$ to every $b_m$ with $m>g(n,s)$ and homogenize all such points, i.e., add cycles so that all points $b_m$, $m>g(n,s)$ have isomorphic connected components.
     \item If $g(n,s)<g(n,s-1)$, homogenize all points $b_m$ for $m>g(n,s)$. For the points $b_m$ with $m\leq g(n,s)$ we want them to contain a cycle that does not appear in the connected component of $b_{g(n,s)+1}$. We thus check for all $m\leq g(n,s)$ if there is a cycle connected to $b_m$ of size different than the cycles connected to $b_{g(n,s)+1}$. If so, then we do nothing. If not, then we add a cycle of length $s$ to $b_m$.
     \item End of construction at stage $s$.
   \end{enumerate}

   For the construction of $\G$ we proceed as in the construction of $\B$ except that we call the new point $a_s$ and add a point $a$ at stage $0$. At each stage $s$ we ensure that the connected component of $a$ is isomorphic to the connected component of $b_s$ at the end of this stage (we may assume that $s>g(n,s)$).
   \begin{claim}\label{cl:atobm}
     The connected component of $a$ is isomorphic to the connected component of $b_m$ for all $m>\liminf_s g(n,s)$.
   \end{claim}
   \begin{proof}
     Say $\liminf_s g(n,s)=k$. Then points $b_s$ with $s>k$ will be homogenized infinitely many times and all be isomorphic. As the connected component of $a$ at a stage $s$ always copies a connected component of a point larger than $g(n,s)$, in the limit its component will be isomorphic to the components of $b_s$ for $s>k$.
   \end{proof}
   \begin{claim}\label{cl:isofin}
     All points $a_m$ and $b_m$ with $m\leq \liminf_s g(n,s)$ have finite connected components. 
     Furthermore, for all $m$, the connected components of $a_m$ and $b_m$ are isomorphic.
   \end{claim}
   \begin{proof}
     Say $\liminf_s g(n,s)=k$. Then there is a least stage $s$ such that $\inf_{t>s} g(n,t)= k$. Say that $s_0$ is the first stage after $s$ such that $g(n,s_0)=k$. Then after this stage all $b_m$ with $m\leq k$ will have a cycle of a length that does not appear among cycles attached to $b_l$ for $l>k$. Furthermore, none of the $b_m$, $m\leq k$ will be homogenized anymore, so they won't obtain new cycles after stage $s_0$. This together with the fact that at each stage only finitely many cycles are added to each component guarantees that all these $b_m$ have finite connected components.
     The last sentence in the claim follows from the symmetry of our constructions for $\A$ and $\B$.  
   \end{proof}
   \cref{cl:isofin} implies directly that $\G\cong \B$. The next claim shows that the degree of categoricity needs to be at least $\mathbf 0''$. 
   \begin{claim}
   Any isomorphism $f:\G\cong \B$ computes $\emptyset''$.
 \end{claim}
 \begin{proof}
 Let $a^n$ be the element corresponding to $a$ in the $n$-slice and $b^n_m$ the elements corresponding to $b_m$ in that slice. Then by \cref{cl:atobm}, $f(a^n)=b_m^n$ for some $m>\liminf_s g(n,s)$ and so $f$ clearly computes the function $h$ such that $f(a^n)=b_{h(n)}^n$. But then $h(n)\geq \liminf_{n,s} g(n,s)$ and as the latter function was the modulus for $\emptyset''$, $f\geq_T h\geq_T \emptyset''$.
 \end{proof}
 That $\A$ has Scott rank $1$ now follows from the following fact.
 \begin{claim}
   Every automorphism orbit in $\A$ is existentially definable.
 \end{claim}
 \begin{proof}
   We will show that the points $a_k^n$ have existentially definable automorphism orbits. This implies that all the elements in the cycles connected to them also have existentially definable automorphism orbits. Let $C_m(x)$ say that $x$ has degree larger than 1 and is part of a cycle of length $m$. Then each of the $a_k^n$ with $k\leq \liminf_s g(n,s)$ satisfies $C_{m_k}(a_k^n)$ for some $m_k$ so that none of the $a_j^n$ with $j>\liminf_s g(n,s)$ satisfies $C_{m_k}(a_k^n)$. Then the automorphism orbit of such an $a_j^n$ is defined by 
   \[ R_n(x)\land \exists y_1\dots y_{g(n,s)} \ C_{m_0}(y_1)\land y_1\neq x\land \dots \land C_{m_{g(n,s)}}(y_{g(n,s)}) \land y_{g(n,s)}\neq x. \]
   On the other hand for the $a_k^n$, $k\leq \liminf_s g(n,s)$, let $D_k(x)$ be the existential formula describing the connected component of $a_k^n$ and let $S$ be the set of indices of points such that the connected component of $a_k^n$ properly embeds into the connected component of these points. 
As all connected components of the $a_s^n$ with $s\in S$ contain a cycle present in at most finitely many connected components in the $n$-slice, $S$ is finite and for every $s\in S$, the connected component of $a_s^n$ is finite. 
    Then the automorphism orbit of $a_k^n$ is defined by the formula
   \[ R_n(x)\land D_k(x) \land \bigwedge_{s\in S} \exists x_s D_s(x_s) \land x_s\neq x.\qedhere\]
 \end{proof}
 This completes the proof of~\cref{lem:sr1dgcat''}.
 \end{proof}
\subsection{Jump inversions and the successor case}
To prove the successor case we will combine two jump inversion techniques with a relativized version of \cref{lem:sr1dgcat''}. We will use \cref{lem:sr1dgcat''} relative to $\Delta^0_{2\alpha+1}$, building a $\Delta^0_{2\alpha+1}$-computable structure $\+A$ of Scott rank $1$ so that every isomorphism between two $\Delta^0_{2\alpha+1}$-computable copies must compute every $\Delta^0_{2\alpha+3}$ set. We then use two different jump inversion, or Marker extension, techniques to obtain the successor step of our induction. Both of these techniques use pairs of structures as pioneered by Ash and Knight, culminating in the material in~\cite{ash2000}. For a detailed development of what we call friendly jump inversions, see for example~\cite[Chapter X.3]{montalban2025}. The other technique, called unfriendly jump inversions, was recently developed in~\cite{harrison-trainor2025}. The two techniques work similarly for the most part, so let us introduce the general machinery.

 Given a graph $\+H$ and two structures $\+A,\+B$ in vocabulary $\tau$, one creates a new structure $\+H^*$ in the vocabulary 
   $\{V/1, W/3, U/1\}\cup \tau$ satisfying
   \begin{enumerate}
     \item\label{it:vertset} $V^{\+H^*}$ is infinite and coinfinite,
     \item\label{it:partition} $\neg V^{\+H^*}$ is partitioned into infinitely many infinite sets $S_{u,v}$ and $W(u,v,-)=S_{u,v}$ for each $u,v\in V^{\+H^*}$.
     \item\label{it:pair} $U^{\+H^*}$ partitions each $S_{u,v}$ into infinite, coinfinite sets.
   \end{enumerate}
   Then, using $E$ to denote the edge relation, we define the substructures
   \[ S_{u,v}\cap U^{\+H^*}\cong \begin{cases} \+A & u \mathrel{E}^\+H v\\ \+B &\neg u\mathrel{E}^\+H v \end{cases}\text{ and }
      S_{u,v}\cap \neg U^{\+H^*}\cong \begin{cases} \+B & u \mathrel{E}^\+H v\\
     \+A & \neg u\mathrel{E}^\+H v \end{cases},\]
     i.e., the substructure on $W(u,v,-)$ is either isomorphic to the pair of structures $(\A,\B)$ or $(\B,\A)$ depending of whether there was an edge between $u$ and $v$ in $\+H$. From the fact that graphs are universal for effective bi-interpretability we can assume without loss of generality that $\+H^*$ is a graph, allowing us to iterate this process. The following claim is rather obvious from the definition of $\+H^*$.
     \begin{lemma}
       Suppose that $\+A$ and $\+B$ are such that every isomorphism between $\+A$ and $\+B$ computes every $\Delta^0_\alpha$ set for some computable ordinal $\alpha$. Then so does every isomorphism between $\+A^*$ and $\+B^*$.
     \end{lemma}
     \begin{proof}
       Note that there are computable bijections $g: \+A\to V^{\+A^*}$ and $h:\+B\to V^{\+B^*}$. Thus, for an isomorphism $f^*:\+A^*\to\+B^*$, $f=h^{-1}\circ f^*\circ g\leq_T f^*$ is an isomorphism between $\+A$ and $\+B$ and hence $f^*$ computes every $\Delta^0_\alpha$ set given that $f$ does.
     \end{proof}
    It was shown in~\cite[Lemma 5.1]{goncharov2005} that if two structures $\+S_0$ and $\+S_1$ are $\alpha$-friendly (i.e., for $\beta\leq \alpha$, their asymmetric $\beta$-back-and-forth relations are uniformly c.e.) and satisfy the same $\Pinf{\alpha}$ sentences, but for each there is a $\Picom{\alpha+1}$ sentence satisfied by it but not the other, then for any $\Delta^0_{\alpha+1}$ set $S$, there is a uniformly computable sequence $(\+C_i)$ so that
       \[ \+C_i\cong \begin{cases} \+S_0 & i\in S\\
                                   \+S_1 & i\not\in S
   \end{cases}.\]
   Using such $\+S_0$ and $\+S_1$ as $\+A$ and $\+B$ gives rise to what we will refer to as \define{$\alpha$-friendly jump inversions}.
   It is straightforward to see from Ash's analysis of the back-and-forth relations on well-orderings~\cite{ash1986} that $\omega^\alpha2 \equiv_{2\alpha} \omega^\alpha$, but $\omega^\alpha2<_{2\alpha+1}\omega^\alpha $. In particular, $\omega^\alpha$ has a $\Pinf{2\alpha+1}$ Scott sentence and $\omega^\alpha2$ has a $\dSinf{2\alpha+1}$ Scott sentence. Thus, the pair $\omega^\alpha$, $\omega^\alpha2$ satisfies the conditions for $2\alpha$-friendly jump inversions. The following fact is obvious from the fact that if $\+H$ is $\Delta^0_{2\alpha+1}$-computable, then so is its edge relation.
   \begin{lemma}\label{claim:jumpinversion}
     Suppose that $\+H$ is $\Delta^0_{2\alpha+1}$-computable. Then there is a computable copy of $\+H^*$, where $\+H^*$ is the friendly jump inversion using $\+A=\omega^\alpha$ and $\+B=\omega^\alpha2$.
   \end{lemma}
   We are now ready to prove the successor case of the hard direction of \cref{thm:highforsr}.
 \begin{lemma}\label{lem:highforsrsucc}
   For $\alpha$ a computable successor ordinal, a degree $\*d$ is high for Scott rank $\alpha$ if $\*d$ computes every $\Delta^0_{2\alpha+1}$ set.
 \end{lemma}
 \begin{proof}
   Suppose that $\alpha=\lambda+n+1$ where $\lambda$ is a limit ordinal. If $n=\lambda=0$, then the lemma follows from \cref{lem:sr1dgcat''}. For the other cases we will carry out the construction in \cref{lem:sr1dgcat''} relativized to $\Delta^0_{(2\lambda+2n+1)}$, i.e., we produce $\Delta^0_{(2\lambda+2n+1)}$-computable structures $\+G$ and $\B$ with Scott rank $1$ so that every isomorphism between $\+G$ and $\B$ computes a complete $\Delta^0_{(2\lambda+2n+3)}$ set. We then combine friendly and unfriendly jump inversions to obtain isomorphic computable structures $\+G^{(-2\lambda+2n+1)}$ and $\B^{(-2\lambda+2n+1)}$ with Scott rank $\alpha$ so that every isomorphism between them computes a complete $\Delta^0_{2\alpha+1}$ set. 
 
   We will use unfriendly jump inversions to obtain $\Delta^0_{2\lambda+1}=\Delta^0_{\lambda+1}$-computable structures $\+G^{(-2n)}$ and $\B^{(-2n)}$ of Scott rank $n$ while ensuring that every isomorphism between them still computes a complete $\Delta^0_{2\alpha+1}$ set. The structures $\+G^{(-2n)}$ and $\B^{(-2n)}$ are obtained from $\+G$ and $\B$ using the techniques outlined above using so-called unfriendly structures. Properties of this jump inversion are summarized in~\cite[Theorem 4.3]{harrison-trainor2025}. In particular, we obtain the following:
   \begin{itemize}
   \item $\+G^{(-2n)}$ and $\B^{(-2n)}$ are $\Delta^0_{2\lambda+1}$-computable
     (\cite[Theorem 4.3 Item 1]{harrison-trainor2025} as $\G$ and $\B$ are $\Delta_{2\lambda+2n+1}$-computable).
   \item $\+G^{(-2n)}$ and $\B^{(-2n)}$ have Scott rank $n$
     (\cite[Theorem 4.3 Item 4]{harrison-trainor2025} as $\G$ and $\B$ have Scott rank $1$, i.e., $\Pi^0_2$ Scott sentences).
   \end{itemize}
   Assuming without loss of generality that they are graphs, we then jump invert $\+G^{-2n}$ and $\B^{-2n}$ using friendly jump inversions with the pairs $(\omega^\alpha,\omega^\alpha2)$, $(\omega^\alpha2,\omega^\alpha)$. That this yields computable structures $\+G^{-\lambda-2n}$ and $\+B^{-\lambda-2n}$ follows from \cref{claim:jumpinversion}. It only remains to show that they have the required Scott rank.

  Let $\theta$ be the sentence formally axiomatizing \cref{it:vertset,it:partition,it:pair} of our jump inversion definition, which has finite quantifier complexity. Then, letting $\phi$ be the $\Pinf{\lambda+1}$ Scott sentence for $\omega^\lambda$ and $\psi$ the $\dSinf{\lambda+1}$ Scott sentence for $\omega^\lambda2$, we can relativize quantification in these sentences to $W(x,y,-)\cap U$ and $W(x,y,-)\cap \neg U$ to obtain formulas $\phi^+(x,y)$, $\psi^+(x,y)$ and $\phi^-(x,y)$, $\psi^-(x,y)$, respectively. Then
  \[ \Theta=\theta\land  \forall x\forall y (\phi^+(x,y) \land \psi^-(x,y)) \lor (\phi^-(x,y)\land \psi^+(x,y))\]
  is $\Pinf{\lambda+2}$ and expresses that any structure satisfying it needs to be a jump inversion of the desired form. At last, take a $\Pinf{n+1}$ Scott sentence $\Gamma$ in prenex normal form for $\+G^{-2n}$. Let $\Gamma^*$ be obtained by relativizing quantifiers to $V$ and replacing basic formulas of the form $x\mathrel{E}y$ with the $\Pinf{\lambda+1}$-formula $\phi^+(x,y)$ or $\Sinf{\lambda+1}$-formula $\neg \phi^-(x,y)$, and formulas of the form $\neg x\mathrel{E}y$ with the formulas $\phi^-(x,y)$ or $\neg\phi^+(x,y)$ so that $\Gamma^*$ is $\Pinf{\lambda+n+1}$. Then the formula $\Theta\land \Gamma^*$ is a $\Pinf{\lambda+n+1}$ Scott sentence for $\+G^{-\lambda-2n}$ showing that $\+G^{-\lambda-2n}$ has Scott rank $\lambda+n=\alpha$ as required.
 \end{proof}
 \subsection{The limit case}
 For limit ordinals $\alpha$, structures with Scott rank $\alpha$ may have either Scott sentence complexity $\Pinf{\alpha}$ or $\Pinf{\alpha+1}$. In \cref{thm:limithighforsr} below we produce a computable structure of Scott sentence complexity $\Pinf{\alpha}$ so that every degree computing isomorphisms between all computable copies must compute every $\Delta^0_{\alpha+1}$ set. This is, colloquially speaking, the optimal result as structures with $\Pinf{\alpha}$ Scott sentence complexity are simpler than structures with $\Pinf{\alpha+1}$ Scott sentence complexity. In order to establish the Scott sentence complexity we use the following internal characterization given in~\cite{gonzalez2024}.
 \begin{lemma}[{\cite[Corollary 2.11]{gonzalez2024}}]\label{lem:unstableseq}
   A structure $\+A$ with Scott rank $\alpha$ has a $\Pinf{\alpha}$ Scott sentence if for every fundamental sequence $(\beta_i)_{i\in\nat}$ for $\alpha$ and complete $\beta_i$-type $p_i(\bar x)$ with $|x|=m$ for some $m$ and $p_i\subset p_{i+1}$, every type $p_i$ being realized in $\+A$ implies that there is a tuple $\bar a$ realizing every type $p_i$.
 \end{lemma}

 \begin{lemma}\label{thm:limithighforsr}For $\alpha$ a computable limit ordinal, there is a computable structure of Scott rank $\alpha$, in particular of Scott sentence complexity $\Pinf{\alpha}$, so that if $\*d$ computes an isomorphism between all its computable copies, then $\*d$ computes every $\Delta^0_{2\alpha+1}$ set.
 \end{lemma}
 \begin{proof}
   A similar construction producing a structure with the same properties already appeared in~\cite{csima2020}. We provide a sketch here and refer the reader to this article for details. The set $\emptyset^{(\alpha+1)}$ is a complete $\Delta^0_{2\alpha+1}$ set and has a modulus $m$ that is limitwise monotonic in $\emptyset^{(\alpha)}$. We will build structures $\+G$ and $\B$ so that every isomorphism between them computes the modulus. Say $l$ gives the $\emptyset^{(\alpha)}$ approximation to $m$. Fixing a fundamental sequence $(\beta_i)_{i\in \nat}$ for $\alpha$ and taking $\emptyset^{(\alpha)}=\bigoplus_i \emptyset^{(\beta_i)}$, notice that every computation from $\emptyset^{(\alpha)}$ can actually be computed from $\emptyset^{(\beta_i)}$ for some $\beta_i$. 
   Hence, invoking Ash and Knight's pairs of structures theorem~\cite{ash1990} for each $\emptyset^{(\beta_i)}$ we obtain a uniformly computable sequence of structures $\+C_{i,j}$ such that
   \[ \+C_{i,j}\cong\begin{cases} \A_i & j\in \emptyset^{(\beta_i)}\\
   \B_i & j\notin \emptyset^{(\beta_i)}\end{cases}.\]
   The structures will be made up of $n$-slices for every $n$ and each slice will contain infinitely many numbered sequences of structures $(\+D_{i,0},\+D_{i,1},\+D_{i,2}\dots)_i$. In the structure $\+G$ we additionally add a sequence of structures $(\+D_{\infty,0},\dots)_\infty$. In both $\B$ and $\+G$ for every $n$-slice we let
   \begin{enumerate}
     \item $\+D_{i,j}\cong \A_j$ if there is $s$ such that $l(n,s)\geq i$ and the computation is bounded by $\emptyset^{(\beta_j)}$,
   \item $\+D_{i,j}\cong \B_j$ otherwise, i.e., if either the computation needs a stronger oracle or $l(n,s)<i$.
   \end{enumerate}
   Additionally, in $\+G$, $\+D_{\infty,j}\cong \B_j$ for all $j$. Now say that $m(n)=k$, then $(\+D_{\infty,j})\cong \+D_{i,j}$ if and only if $i>k$. Hence, any isomorphism $f:\+G\to \B$ must send $(\+D_{\infty,j})$ to some $(\+D_{i,j})$ with $i>k$. In particular, any such isomorphism can compute the index $i$ uniformly in $n$, thus it computes a function dominating $m$.

   For $\+A_i$ and $\+B_i$ we can use $\omega^{\beta_i}$ and $\omega^{\beta_i}2$, respectively, as in the proof of \cref{lem:highforsrsucc}.
   This guarantees that $\+G$ has Scott rank $\alpha$. As promised above we can in fact establish $\+G$'s Scott sentence complexity.
     \begin{claim} $\+G$ has Scott sentence complexity $\Pinf{\alpha}$.\end{claim}
     \begin{proof}
       We will use \cref{lem:unstableseq}. Suppose that $(p_i)_{i\in\nat}$ is a sequence of types realized in $\+G$ as in the lemma. Then notice that if $\bar a$ and $\bar b$ satisfy any type in this sequence, then for every $k<|\bar a|$, $a_k$ and $b_k$ need to be in the same $n$-slice $n_0$ and in the same column of a sequence in this $n$-slice. As the structures $\omega^{\beta_i}$ and $\omega^{\beta_i}2$ are rigid and have Scott sentence complexity below $\Pinf{\alpha}$, there is $i_0$ such that $p_{i_0}$ specifies the isomorphism type of the columns containing the free variables, and the rank of the elements within the orderings. In particular, $p_{i_0}$ specifies the isomorphism types of the columns to the left of any free variable. If the element in the rightmost column containing a free variable is in a linear ordering isomorphic to $\+B_j$, then $p_{i_0}$ already specifies the automorphism type of the tuple, hence the tuple realizing $p_{i_0}$ realizes all $p_i$. Otherwise, by choosing $i_1>i_0$ large enough so that $p_{i_1}$ specifies the isomorphism type of a column $j$ isomorphic to $\+B_j$, one easily sees that $p_{i_1}$ specifies the automorphism orbit of the tuple realizing it and hence this tuple realizes all types $(p_i)_{i\in\nat}$.
     \end{proof}
     This completes the proof of~\cref{thm:limithighforsr}.
 \end{proof}
 
 \section{Highness for computably defined Scott ranks}\label{sec:cdSr}
 The goal of this section is to prove the following theorem.
 
 \begin{theorem}\label{thm:highforcompSR}
    Suppose $0 < \alpha < \ock$,
\begin{enumerate}
\item If $\alpha$ is a limit ordinal, then $\*d$ is high for computably defined Scott rank $\alpha$ if and only if $\*d$ computes every $\Delta^0_{\alpha+1}$ set.
\item If $\alpha$ is a successor ordinal, then $\*d$ is high for computably defined Scott rank $\alpha$ if and only if $\*d$ computes every $\Delta^0_{\alpha+2}$ set.
\end{enumerate}
\end{theorem}
\begin{proof}
  \emph{Right-to-left direction of item 2.} First note that the index set of $\Sicom{\alpha}$-formulas is computable (see~\cite[p.~106]{ash2000}) and satisfaction of $\Sicom{\alpha}$-formulas in computable structures is $\Sigma^0_\alpha$~\cite[Theorem 7.5]{ash2000}. For a given computable structure $\+A$ and fixed $n\in\nat$ consider the lattice of $\Sicom{\alpha}$-definable subsets of $\A^n$ ordered by subset, i.e.,
  \[ \mathbb{P}_{\Sicom{\alpha},n}^{{\+A}}=\{ \phi^\+A: \phi(x_1\dots x_n)\in \Sicom{\alpha}\}.\]
  If $\+A$ has Scott rank $\alpha$, then for an n-tuple $\bar a$, the formula $\phi$ defining the automorphism orbit of $\bar a$ is one so that $\phi^{\+A}$ is an atom in $\mathbb{P}_{\Sicom\alpha,n}^{\+A}$, i.e., for all $D\in \mathbb{P}_{\Sicom\alpha,n}^{\+A}$, $D\land \phi^{\+A}=\emptyset$ or $D\land \phi^{\+A}=\phi^{\+A}$. A degree $\*d$ computing every $\Delta^0_{\alpha+2}$ set can compute the atoms in $\mathbb{P}_{\Sicom\alpha,n}^{\+A}$ as well as in the corresponding partial ordering for $\+A$ expanded by constants, and thus, given two computable copies of $\+A$, can compute a back-and-forth set, hence an isomorphism between them.

  \emph{Right-to-left direction of item 1.} Fix a fundamental sequence $(\beta_i)_{i\in\nat}$ for $\alpha$. We will look at the algebras $\mathbb P_{\Sicom{\beta_i},n}^{{\+A}}$ defined as above. To compute the formula defining the automorphism orbit of an n-tuple $\bar a$ we have to find $i$ and $S\in\mathbb P_{\Sicom{\beta_i},n}^{{\+A}}$ so that $\bar a\in S$ and for all $j\geq i$, $S$ is an atom in $S\in\mathbb P_{\Sicom{\beta_i},n}^{{\+A}}$. We claim that any formula defining $S$ defines the automorphism orbit of $\bar a$. Since $\+A$ has computably defined Scott rank $\alpha$, there is $i$ such that the automorphism orbit of $\bar a$ is $\Sicom{\beta_i}$-definable. Hence, for $j\geq i$, if $\phi^{\+A}$ is an atom in $\mathbb P_{\Sicom{\beta_j},n}^{{\+A}}$, then $\phi$ defines $\bar a$'s automorphism orbit. Thus, by definition, any formula defining $S$ must define $\bar a$'s automorphism orbit. Counting quantifiers and using the uniformity of this argument we conclude that for any two computable copies of $\+A$, there is a $\Pi^0_\alpha$ back-and-forth set and thus a $\Delta^0_{\alpha+1}$ computable isomorphism.

  \emph{Left-to-right direction of item 1.} This follows from the fact that the structure produced in \cref{thm:limithighforsr} has computably defined Scott rank. To see this, recall that the automorphism orbit of an element in this structure is defined by (1) its slice, (2) the position of its ordering in its sequence of orderings, (3) the isomorphism type of the sequence it is in, and (4) its rank in its ordering. Items (1)--(2) can be defined using quantifier-free formulas, but for items (3)--(4) we need formulas of higher complexity. However, as the rank of an element in its ordering is definable by a computable infinitary formula and the isomorphism type of the sequence is determined by a finite initial segment, these formulas can be taken to be computable of complexity less than $\Sinf\alpha$. A similar argument works for finite tuples and thus it follows that $\+G$ has computably defined Scott rank $\alpha$, as required.
  
  \emph{Left-to-right direction of item 2.} For the base case note that the structure produced in \cref{lem:sr1dgcat''} of Scott rank $1$ has computably defined Scott rank $1$. In fact, this holds for any structure of Scott rank $1$ as their automorphism orbits are definable by finite existential formulas. For successor $\alpha=\lambda+k$, the proof goes along the same lines as the proof of \cref{lem:highforsrsucc} except that we only use friendly jump inversions. Relativize the base case construction to a $\Delta^0_{\alpha}$-complete set $S$. Then there are $S$-computable structures $\+G$ and $\+B$ of Scott rank $1$, so that any isomorphism between them computes every $\Delta^0_{\alpha+2}$ set. Now we jump invert using linear orderings as follows. If $k=2n$, then we do a friendly jump invert with the pair $\omega^{\lambda+n}$, $\omega^{\lambda+n}2$. By \cref{claim:jumpinversion} we obtain computable structures $\+G^{-\lambda-2n}$ and $\+B^{-\lambda-2n}$; that they have Scott rank $\alpha$ follows from the same argument as in the last paragraph of the proof of \cref{thm:limithighforsr}. On the other hand, if $k=2n+1$, then we use the successor orderings of order-type $\omega^{\lambda+n+1}$ and $\omega^{\lambda+n+1}2$. As linear orderings, $\omega^{\lambda+n+1}\equiv_{\lambda+2n+2} \omega^{\lambda+n+1}2$ and $\omega^{\lambda+n+1}2<_{\lambda+2n+3} \omega^{\lambda+n+1}$, and since successor orderings are the canonical structural jumps of linear orderings~\cite[Chapter 10]{montalban2021}, as successor orderings
\[ (\omega^{\lambda+n+1},s)\equiv_{\lambda+2n+1} (\omega^{\lambda+n+1}2,s) \text{ and }(\omega^{\lambda+n+1}2,s)<_{\lambda+2n+2} (\omega^{\lambda+n+1},s)\]
  so that jump inverting with $(\omega^{\lambda+n+1},s)$, $(\omega^{\lambda+n+1}2,s)$ yields the required computable structures $\+G^{-\lambda-2n-1}$ and $\+B^{-\lambda-2n-1}$. That they have Scott rank $\alpha$ again follows from the argument in the last paragraph of \cref{lem:highforsrsucc}.  
\end{proof}

\subsection{A coincidence of theorems}

Note that if $\alpha$ is a limit, then $2\alpha+1 = \alpha+1$.  Further, if $\alpha$ is 1 or the successor of a limit, then $2\alpha+1 = \alpha+2$.  Thus for either of these cases, \cref{thm:highforsr,thm:highforcompSR} coincide: a degree $\*d$ is high for Scott rank $\alpha$ if and only if $\*d$ is high for computably defined Scott rank $\alpha$.  This is explained by \cref{thm:limitiffcompdefined}, which shows that for such an $\alpha$, a computable structure has Scott rank $\alpha$ if and only if it has computably defined Scott rank $\alpha$.

\section{Index sets and coincidence at limits}\label{sec:index}
In this section we calculate the complexity of the index sets of computable structures with Scott rank $\alpha$ and computably defined Scott rank $\alpha$ (\cref{thm:indexsets}). It is perhaps surprising that these complexities coincide for all computable ordinals. Of its own interest, and central in explaining some of the other results in the paper, is \cref{thm:limitiffcompdefined}, which shows that for limit ordinals every structure with Scott rank at that level has computably defined Scott rank. We prove these results in \cref{subsec:indexsetcalculations} after we prove a modification of Ash and Knight's pairs of structures theorem, which we will need to show that our lower bounds are tight.
  \subsection{Limits of structures}\label{subsec:limits}
Our main tool in this section are $\alpha$-systems as developed by Ash and Knight, see~\cite{ash2000} for a detailed discussion. For computable $\alpha$, an $\alpha$-system is a structure
\[ (L,U,\hat l, P,E, (\leq_\beta)_{\beta<\alpha})\]
having the following properties. $L$ and $U$ are c.e.\ sets and $P$ is a c.e. \emph{alternating tree} on $L$ and $U$ starting at $\hat l\in L$. That means that elements of $P$ are of the form $\sigma=\hat lu_0l_0u_1l_1,\dots$ where $u_i\in U$ and $l_i\in L$. We assume that every tree $P$ satisfies that every element of it has a proper extension in $P$. $E$ is a partial computable function from $L$ to the set of finite subsets of the natural numbers and the $\leq_\beta$ are binary relations on $L$, c.e.\ uniformly in $\beta$. Now for this structure to be an $\alpha$-system it additionally has to satisfy the following properties:
\begin{enumerate}
    \item\label{it:asys1}$\leq_\beta$ is reflexive and transitive, for $\beta<\alpha$,
    \item\label{it:asys2}$\leq_\beta$ is \emph{nested}, i.e.\ , $l\leq_\gamma l'\Ra l\leq_\beta l'$ for $\beta<\gamma<\alpha$,
    \item\label{it:asys3} if $l\leq_0 l'$, then $E(l) \subseteq E(l')$,
    \item\label{it:asys4} if $\sigma u\in P$, where $\sigma$ ends in $l_0\in L$, and
    \[ l_0\leq_{\beta_0} l_1\leq_{\beta_1}\dots\leq_{\beta_{k-1}} l_k,\]
    for $\alpha>\beta_0>\beta_1>\dots >\beta_k$, then there exists $l^*$ such that $\sigma ul^*\in P$, and $l_i\leq_{\beta_i} l^*$, for all $i\leq k$.
\end{enumerate}
An \emph{instruction function} $q$ is a function from sets of sequences in $P$ with last terms in $L$ to $U$ such that $q(\sigma)=u$ implies $\sigma u\in P$. A \emph{run} of $(P,q)$ is a path $\pi=\hat l u_1l_1u_2l_2\dots$ through $P$ such that for all $n\in \omega$
\[ u_{n+1}=q(\hat lu_1l_1\dots u_nl_n).\]
Ash proved the following powerful theorem, called the metatheorem for $\alpha$-systems. We state it here but refer the reader to~\cite{ash2000} for a proof.
\begin{theorem}[{\cite[Theorem 14.1]{ash2000}}]
Let $(L,U,\hat l, P,E, (\leq_\beta)_{\beta<\alpha})$ be an $\alpha$-system. Then for any $\Delta_\alpha^0$ instruction function $q$, there is a run $\pi$ of $(P,q)$ such that $E(\pi)$ is c.e., while $\pi$ itself is $\Delta_\alpha^0$. Moreover, from a $\Delta_\alpha^0$ index for $q$, together with a computable sequence of c.e.\ and computable indices for the components of the $\alpha$-system, we can effectively determine a $\Delta_\alpha^0$ index for $\pi$ and a c.e.\ index for $E(\pi)$.
\end{theorem}
In order to use this theorem to prove our results we need another definition related to back-and-forth relations.
\begin{definition}\label{def:alphafriendly}
A finite or countable sequence $\langle \A_0,\A_1,\dots \rangle$ of structures is \define{$\alpha$-friendly} if the structures $\A_i$ are uniformly computable and for $\beta<\alpha$ the back-and-forth relations $\leq_\beta$ on the set of pairs $(\A_i, \ol a)$ for $\ol a\in A_i^{<\omega}$ are c.e., uniformly in $\beta$.
\end{definition}

\begin{lemma}\label{thm:alphalimit}
Let $( \mc A_k )_{k\in \omega}$ be an $\alpha$-friendly family of structures such that $\mc A_k\subseteq \mc A_{k+1}$ and $\mc A_{k+1}\leq_{\alpha} \mc A_k$ for all $k$. Then for any $\Sigma_{\alpha+1}^0$ set $S$ such that $n\in S$ if and only if $\exists x \langle x,n\rangle \in P$ for a $\Pi_{\alpha}^0$ set $P$ there is a uniformly computable sequence of structures $(\mc C_n)_{n\in \omega}$ such that
\[ \mc C_n \cong \begin{cases} \mc A_i,\ i =\mu y[\langle y,n\rangle \in P] & \text{if } \exists x \langle x,n\rangle \in P\\
\lim_i \mc A_i & \text{otherwise}
\end{cases}.\]
\end{lemma}
\begin{proof}
As $P$ is $\Pi_{\alpha}^0$ we have that there is a $\Delta_{\alpha}^0$ computable function $g$ such that
\begin{align*}
    \langle x,n\rangle \in P &\iff \forall t\ g(\langle x,n\rangle, t)=0\\
    \langle x,n\rangle \not \in P &\iff \exists t\ g(\langle x,n\rangle, t)=1
\end{align*}

We will define the $\alpha$-system $(L,U,\hat l, P,E,(\leq_\beta)_{\beta<\alpha})$ and an instruction function $q_n$ for each $n\in \omega$. We fix $n$. Let $C$ be an infinite set of constants for the universe of $\mc C_n$ and $\mc F$ be the set of finite partial injective functions $p$ from $C$ to $A=\bigcup_{k\in \omega} A_k$. Further let $U=\omega$ and $L=\omega\times \mc F \cup \{(-1,\emptyset)\}$. We let $\hat l=(-1,\emptyset)$.

For every $m\in \omega$, every $L$-structure $\mc B$ and $\ol b\in B^{m}$, the \emph{standard enumeration function} $E_{st}$ maps the pair $(\mc B,\ol b)$ to the atomic formulas or negations thereof $\psi(\ol x)$ with G\"odel numbers less than $m$ such that $\mc B\models \psi(\ol b)$. We extend the standard enumeration function to define $E$. Let $E(\hat l)=\emptyset$ and if $l=(i,p:\ol b\ra \ol a)$, then $E(l)=E_{st}(\mc A_i, \ol a)$.

If $l=(i,p:\ol b \ra \ol a)$ and $l'=(j,q:\ol b'\ra \ol a')$ then $l\leq_\beta l'$ if and only if $\ol b\subseteq \ol b'$ and $(\mc A_i,\ol a)\leq_\beta (\mc A_j,\ol a')$. We write $l\subseteq l'$ if $i=j$ and $p\subseteq q$, and define for all $l$ and $\beta<\alpha$, $\hat l\leq_\beta l$.

Our tree $P$ consists of the finite alternating sequences
\[ \sigma=\hat lu_1l_1u_2l_2\dots\]
where $u_k\in U$, $l_k\in L$, and the following conditions hold:
\begin{enumerate}[label=(\alph*)]
    \item $l_k=(u_k,p_k)$,
    \item\label{it:posmodified} $dom(p_k)$ and $ran(p_k)$ include the first $k$ elements of $C$ and $A_{u_k}$,
    \item if $u_k\neq u_{k+1}$, then $u_{k+1}\geq u_k$,
    \item if $u_k=u_{k+1}$, then $l_n\subseteq l_{n+1}$, and in any case $l_k\leq_0 l_{k+1}$.
\end{enumerate}
\begin{claim}
 The structure $(L,U,\hat l, P,E,(\leq_\beta)_{\beta<\alpha})$ is an $\alpha$-system.
\end{claim}
\begin{proof}
We have to check the $4$ properties from the definition of an $\alpha$-system. \cref{it:asys1,it:asys2,it:asys3} are easy to see. We turn to \cref{it:asys4}. Assume $\sigma u\in P$, where $\sigma$ ends in $l_0$, and
\[ l_0\leq_{\beta_0}\dots \leq_{\beta_{k-1}} l_k,\]
for $\alpha>\beta_0>\dots>\beta_k$. Say $l_0=(v,p)$. We have to find $l$ such that $l_i\leq_{\beta_i} l$ for $0\leq i\leq k$. If $u=v$, then by \cref{it:asys1,it:asys2,it:asys3} we can find $q\supset p$ such that $l_i \leq_{\beta_i} l=(v,q)$ for all $0\leq i\leq k$. If $u>v$, then as $\A_{u} \supseteq \A_v$ and $\A_{u}\leq_{\alpha}\A_{v}$ we get that there is $C\ra A_u:q \supset p$ such that $l_i\leq_{\beta_i}l=(u,q)$ for all $0\leq i\leq k$. Thus our system satisfies all properties of an $\alpha$-system.
\end{proof}
We can now define the instruction functions $q_n$.
\[ q_n(\sigma)=\begin{cases}
u_s & \text{if } \forall t< |\sigma|\ g(\langle u_s,n\rangle,t)=0\\
u_s+1 &\text{otherwise}
\end{cases}\]
As $q_n$ solely depends on $g$ we can compute a $\Delta_\alpha^0$ index for $q_n$. Thus applying the metatheorem we get a run $\pi_n=\hat l u_1^n l_1^n u_2^n l_2^n\dots$ of $(P,q_n)$ for each $n$ such that $E(\pi_n)$ is c.e., uniformly in $n$.
Assume that $l_s^n=(u_s^n,p_s^n)$, and let $F$ be defined by $F^{-1}=\bigcup_{s\in \omega} p_s^{n}$. Define $\mc C$ as the pullback structure of $F$. As $q_n$ is monotonic and $\A_i\subseteq \A_{i+1}$ we get that if $\langle x,n\rangle\in P$ there is an $s_0$ such that for all $s>s_0$ the system builds an embedding between $\mc C$ and $\mc A_n$, as for all $i<n$ $\A_i\subseteq \A_n$. By similar reasons if $\langle x,n\rangle\not \in P$ we build a structure embeddable in the limit. \cref{it:posmodified} guarantees that the built embeddings are isomorphisms.
\end{proof}
\begin{lemma}\label{lem:ordinalshard}
  Let $\alpha$ be a non-zero even computable ordinal. Then for any $\Sigma^0_{2\alpha+2}$ set $S$, there is a uniformly computable sequence of linear orderings $(\+C_i)$ such that
  \[ \+C_i\cong\begin{cases} n \omega^\alpha & \text{for some $n\in \nat$ if $i\in S$}\\
  \omega^{\alpha+1} & \text{if } i\not\in S\end{cases}.\]
\end{lemma}
\begin{lemma}\label{lem:succordinalshard}
  Let $\alpha$ be an odd computable ordinal. Then for any $\Sigma^0_{2\alpha+2}$ set $S$, there is a uniformly computable sequence of successor linear orderings $(\+C_i)$ such that
  \[ \+C_i\cong\begin{cases} n\omega^\alpha & \text{for some $n\in \nat$ if $i\in S$}\\
  \omega^{\alpha+1}& \text{if } i\not\in S\end{cases}.\]
\end{lemma}

\subsection{Applications}\label{subsec:indexsetcalculations}
\begin{lemma}\label{lem:alphato2alpha}
  Suppose that $\beta < \ock$ and that the automorphism orbit of $\ba\in \A$ is $\Sinf{\beta}$-definable. Then it is $\Sicom{2\beta+1}$-definable.
\end{lemma}
\begin{proof}
 Suppose that $\bar a$ has $\Sinf{\beta}$ definable automorphism orbit. Then the elements $\bar b$ such that $\bar b\leq_\beta\bar a$ are precisely the elements in $\bar a$'s automorphism orbit. So having fixed $\bar a$, consider the labeled well-founded tree $T$ obtained from the $\beta$-back-and-forth game.

  The root node is labeled with $\bar a$ and the ordinal $\beta$. For readability, if $\nu$ is a node on the tree, let $\nu_i=\nu\restrict i$. Then, nodes $\nu$ at even levels are labeled with the plays of the $\forall$-player, i.e., decreasing ordinals $o(\nu)<\beta$ and tuples $t(\nu)$, while nodes at odd levels are labeled with responses of the $\exists$-player, i.e., tuples $t(\nu)$ such that
  \[
    D(t(\nu_1),t(\nu_3),t(\nu_5),t(\nu_7),\dots)=D(t(\nu_0)=\bar a,t(\nu_2),t(\nu_4),t(\nu_6), \dots).
  \]
  Here $D$ denotes the atomic diagram of the tuple in $\+A$.
  
  Note that $T$ has rank $2\beta$. 
  Now, $\bar b\in aut(\bar a)$ if and only if the $\exists$-player has a winning strategy in the $\beta$-back-and-forth game $\bar b\leq_\beta \bar a$ if and only if for the unique node $\nu$ at level $1$ with $t(\nu)=\bar b$, $\nu\in \mathrm{WS}_\beta$ where $\mathrm{WS}_\beta$ is defined inductively as follows:
  \begin{align*}
    \mathrm{WS}_0&=\{ \nu\in T: \forall x (\nu\concat x\in T \land o(\nu\concat x)=0)\implies \exists y\ \nu\concat x\concat y\in T\}\\
    \mathrm{WS}_{\alpha+1}&=\{ \nu\in T: \forall x (\nu\concat x\in T \land o(\nu\concat x)=\alpha)\implies \exists y\ \nu\concat x\concat y\in \mathrm{WS}_\alpha\}\\
    \mathrm{WS}_{\alpha}&=\bigcap_{\gamma<\alpha} \mathrm{WS}_\gamma \text{ if $\alpha$ limit}
  \end{align*}
  Intuitively, $\nu$ is in the set $\mathrm{WS}_{\alpha}$ for some $\alpha$ if whenever at the current game state specified by $\nu$, the $\forall$-player plays the ordinal $\alpha$ in the next round, then the $\exists$-player has a winning strategy.
  A simple transfinite induction shows that $aut(\bar a)=\{t(\nu): |\nu|=1\land \nu\in \mathrm{WS}_\beta\}$ and that this set is $\Pi^0_{2\beta}$. It is furthermore trivially automorphism invariant and thus relatively intrinsically $\Pi^0_{2\beta}$. By a classical result of Ash, Knight, Manasse, and Slaman~\cite{ash1989}, and independently Chisholm~\cite{chisholm1990}, this implies that $aut(\bar a)$ is $\Picom{2\beta}$-definable, and thus also $\Sicom{2\beta+1}$-definable as required.
\end{proof}
\begin{theorem}\label{thm:limitiffcompdefined}
  Suppose $\lambda\leq \ock$ is a limit ordinal. A computable structure $\+A$ has Scott rank $\lambda$ if and only if $\+A$ has computably defined Scott rank $\lambda$. Furthermore, if $\lambda$ is computable, then the equality also holds for $\lambda+1$.  Finally, it also holds at 1.
\end{theorem}
\begin{proof}
For the left-to-right direction, we first dispense with the case for Scott rank 1.  Suppose that an automorphism orbit is defined by a $\Sinf{1}$ formula.  By passing to the relevant disjunct, we may assume this formula is finite, and thus $\Sicom{1}$.

Now if $\+A$ has Scott rank a limit ordinal $\lambda$, then every automorphism orbit is definable by a $\Sinf{\beta}$ formula for some $\beta<\lambda$. By \cref{lem:alphato2alpha} it is also definable by a $\Sicom{2\beta+1}$-formula. As $2\beta+1<\lambda$ for $\beta<\lambda$, $\+A$ has computably defined Scott rank $\lambda$.

For the case that $\+A$ has Scott rank $\lambda+1$, let $\bar a$ have $\Sinf{\lambda+1}$ definable automorphism orbit. Then without loss of generality there is $\phi\in \Pinf{\lambda}$ such that $\exists \bar y \phi(\bar x,\bar y)$ defines the automorphism orbit of $\bar a$. Pick $\bar c$ such that $\phi(\bar a,\bar c)$. Then for every $\bar b$, $\bar b\in aut(\bar a)$ if and only if there is a $\bar d$ with $\bar a,\bar c\leq_{\lambda}\bar b,\bar d$. Constructing the tree and $\mathrm{WS}_\lambda$ similar to the tree and $\mathrm{WS}_\beta$ above, we get that 
  \[ \bar b\in aut(\bar a) \iff \exists \bar d\exists \nu \ t(\nu)=\bar b\bar d \land |\nu|=1 \land \nu \in \mathrm{WS}_{\lambda}.\] 
  As argued previously, $\mathrm{WS}_\lambda$ is $\Pi^0_{2\lambda}$, and $2\lambda = \lambda$ since $\lambda$ is limit. It follows that $aut(\bar a)$ is relatively intrinsically $\Sigma^0_{\lambda+1}$ and thus $\Sicom{\lambda+1}$-definable.

  For the right-to-left direction, if $\+A$ has computably defined Scott rank $\lambda$, then it has Scott rank at most $\lambda$, so we must only show that it cannot have lower Scott rank. This is trivial for $\lambda=1$.
  For the case where $\lambda$ is a limit ordinal we prove the contraposition. If $\lambda$ is a limit ordinal and $\+A$ has Scott rank $\beta<\lambda$, then by \cref{lem:alphato2alpha} all orbits are $\Sicom{2\beta+1}$ definable. But $2\beta+1<\lambda$ and thus $\+A$ has computably defined Scott rank $2\beta+1<\lambda$. The same argument works for the case where $\+A$ has computably defined Scott rank $\lambda+1$.
\end{proof}
\begin{theorem}\label{thm:indexsets} 
    For all $\alpha<\ock$, the following index sets are $\Sigma^0_{2\alpha+2}$-complete.
    \begin{enumerate}
      \item The index set of structures with Scott rank $\alpha$.
      \item The index set of structures with computably defined Scott rank $\alpha$.
    \end{enumerate}
  \end{theorem}
  \begin{proof}
    To see membership, note that a structure $\+A$ has Scott rank $\alpha$ if and only if there is a parameter such that, after adding the parameter, all $\Pinf{\alpha}$-types are $\Sinf{\alpha}$ isolated, in terms of back-and-forth relations this translates to
    \begin{equation}\tag{$\ast$}\label{eq:sralpha} \exists \bar p \forall \bar a \forall \bar c\  \bar p\bar c\leq_\alpha\bar p\bar a \to \bar a\leq_\alpha \bar c. 
    \end{equation}
    As $\leq_\alpha$ is $\Picom{2\alpha}$, \cref{eq:sralpha} is $\Sigma^0_{2\alpha+2}$ when evaluated in computable structures.
    
    Recall that a structure has computably defined Scott rank $1$ if and only if it has Scott rank $1$. Thus, \cref{eq:sralpha} works in this case. For $\alpha=\lambda$ or $\alpha=\lambda+1$ where $\lambda$ is a limit, we similarly have that every structure of Scott rank $\alpha$ has computably defined Scott rank $\alpha$ by \cref{thm:limitiffcompdefined} and thus \cref{eq:sralpha} works for this case. For all other ordinals, a structure has computably defined Scott rank $\alpha$ if and only if
    it satisfies
    \begin{equation}\label{eq:compsralpha}\tag{$\star$}
    \begin{aligned} \exists \bar p \forall \bar a (\exists \phi\in \Sicom{\alpha}) (\forall \psi\in \Sicom{\alpha}) \forall \bar c &\left(\phi^\+A(\bar p,\bar a) \land \phi^\+A(\bar p, \bar c) \right)\to \left(\psi^\+A(\bar p, \bar a)\leftrightarrow \psi^\+A(\bar p, \bar c)\right)\\
    \land(\forall \psi\in \Sicom{\alpha} )\forall \bar c & \left(\psi^\+A(\bar p, \bar a)\leftrightarrow \psi^\+A(\bar p, \bar c)\right)\to \bar p\bar a \leq_\alpha \bar p\bar c\end{aligned}
    \end{equation}
    The set of $\Sicom{\alpha}$ formulas is decidable, and evaluation of $\Sicom{\alpha}$ formulas in computable structures is $\Sicom{\alpha}$. Given that $\leq_{\alpha}$ is $\Picom{2\alpha}$, evaluating \cref{eq:compsralpha} in computable structures is clearly $\Sigma^0_{2\alpha+2}$ if $\alpha>\lambda+1$ where $\lambda$ is a limit ordinal.

  The hardness of both sets follows from \cref{lem:ordinalshard,lem:succordinalshard} and the fact that for computable $\alpha$ and $n\in\nat$, the ordinals $n\omega^\alpha$ have computably defined Scott rank $2\alpha$, while $\omega^{\alpha+1}$ does not. By adding the successor relation to the vocabulary we obtain examples for odd ranks.
  \end{proof}

 \section{Highness for noncomputable Scott ranks}\label{sec:ncSr}
Here we briefly discuss highness for Scott ranks $\ock$ and $\ock+1$. By \cref{thm:limithighforsr} a computable structure has Scott rank $\ock$ if and only if it has computably defined Scott rank $\ock$. On the other hand, structures of computably defined Scott rank $\ock+1$ do not exist, since there are no $\Sicom{\ock}$ or $\Sicom{\ock+1}$ formulas.

As mentioned before, if $\*d$ is high for Scott rank $\ock+1$, then $\*d$ will compute an isomorphism between any two isomorphic computable structures of Scott rank at most $\ock+1$, i.e., all isomorphic computable structures.  Thus highness for Scott rank $\ock+1$ is simply highness for isomorphism as studied in \cite{cft23}.

The situation for Scott rank $\ock$ is less clear.  Since highness for Scott rank $\ock$ implies highness for Scott rank $\alpha$ for every $\alpha < \ock$, it follows by \cref{thm:highforsr} that a degree which is high for Scott rank $\ock$ must compute all $\Delta^1_1$ sets.  We will show more.

\begin{theorem}
If $\*d$ is high for (computably defined) Scott rank $\ock$, then $\*d$ computes a descending sequence through any ill-founded computable linear order.
\end{theorem}

A degree with this property is called \define{high for descending sequences} in \cite{cft23}, where it was observed that such a degree $\*d$ must have $\*d^{(3)} \ge_T \+O$.

\begin{proof}
Fix an ill-founded computable linear order $\+L$.  By replacing $\+L$ with $1+\omega\cdot \+L+2$, we may assume that the set of limit points and the successor relation on $\+L$ are both computable, that $\+L$ has a distinguished least element, and that $\+L$ has last elements $a < b$.  Note that a descending sequence through $1+\omega\cdot \+L+2$ computes a descending sequence through $\+L$.

We mimic the construction of a thin tree from \cite{km2010}.  First we define an increasing sequence of c.e.\ sets, $(A_n)_{n \in \{-1\} \cup \omega}$.  Considering each $A_n \subseteq \+L$, with the induced order, we will maintain the following property: if $x \in A_n \setminus A_{n-1}$, then either $x$ is the least element of $A_n$ or $x$ has an immediate predecessor in $A_n$, and this predecessor relation is uniformly computable.  Note that $x$'s predecessor in $A_n$ need not be the same as $x$'s predecessor in $\+L$, and indeed $x$ might be a limit point in $\+L$.  We also maintain that $A_n \setminus A_{n-1}$ contains an element of the ill-founded part of $\+L$.
\begin{itemize}
\item $A_0 = \{a, b\}$ and $A_{-1} = \emptyset$.  Here $a$ is the least element of $A_0$ and is $b$'s predecessor within $A_0$.
\item We enumerate all of $A_n$ into $A_{n+1}$.  Also, for $x \in A_n \setminus A_{n-1}$, we may enumerate additional elements depending on $x$'s role in $\+L$:
\begin{itemize}
\item If $x$ is the least element of $\+L$, then we do nothing for $x$.
\item If $x$ is the successor of $y$ in $\+L$, then we enumerate $y$ into $A_{n+1}$.  If $x$ was the least element of $A_n$, then $y$ is the least element of $A_{n+1}$.  If instead $x$ had a predecessor $z$ in $A_n$ other than $y$, then $z$ is $y$'s predecessor in $A_{n+1}$.  (If $y \in A_n$, then we will not have $y \in A_{n+1}\setminus A_n$, and so we do not need to define its predecessor in $A_{n+1}$.)
\item If $x$ is a limit point of $\+L$, then we fix a computable increasing sequence $y_0 < y_1 < \dots$ with limit $x$.  If $x$ had a predecessor $z$ in $A_n$, then we require that $z < y_0$.  For each $i$, $y_i$ is $y_{i+1}$'s predecessor in $A_{n+1}$.  If $z$ exists, it is $y_0$'s predecessor in $A_{n+1}$; otherwise, $y_0$ is the least element of $A_{n+1}$.
\end{itemize}
\end{itemize}
This completes the construction of the sequence $(A_n)_{n \in \{-1\}\cup\omega}$.  A straightforward induction verifies that our predecessors are correctly defined, that $A_n\setminus A_{n-1}$ contains an element of the ill-founded part of $\+L$, and that the order-type of $A_n \subset \+L$ is an ordinal at most $\omega^n\cdot 2$.

Now we form a tree $T$ consisting of decreasing sequences through $\+L$ which meet each $A_n$ in order, and such that all nodes occur with infinite multiplicity.  More formally, $T$ consists of sequences $\sigma \in (\+L \times \omega)^{<\omega}$ satisfying:
\begin{itemize}
\item If $\sigma(n) = (x, k)$, then $x \in A_n$; and
\item If $\sigma(n) = (x, k)$ and $\sigma(n+1) = (y, m)$, then $y <_{\+L} x$.
\end{itemize}
Note that the purpose of the second coordinate is to give all children infinite multiplicity.

A simple induction on tree rank shows that for $\sigma \in T$ non-empty with last element $x$, the tree rank of $\sigma$ is the order-type of $\+L(< x)$ when $x$ is in the well-founded part of $\+L$, and is $\infty$ when $x$ is in the ill-founded part of $\+L$.  Thus $T$ is {\em rank homogeneous}: if $\sigma \in T$ with tree rank $\alpha$, and there is some $\tau \in T$ with tree rank $\beta < \alpha$ and $|\tau| = |\sigma|+1$, then $\sigma$ has infinitely many children of tree rank $\beta$.

We make a computable structure out of $T$ using a unary function which maps the root to itself, and which maps each non-root to its parent.  We define $T_a$ to be the subtree consisting of elements comparable with $\langle (a, 0)\rangle$, and similarly $T_b$.  Then $T_a$ and $T_b$ are isomorphic and have Scott rank at most $\ock$ (Lemmas 3.1 and 3.2 of \cite{ckm06}, respectively).

Now, suppose $f: T_b \to T_a$ is an isomorphism.  Then note that $f$ must map the root to the root, and must send $\langle (b,0)\rangle$ to $\langle (a,0) \rangle$.  Also, $\langle (b, 0), (a, 0)\rangle \in T_b$.  Consider now $f( \langle (b, 0), (a,0) \rangle) = \langle (a, 0), (c,k) \rangle$.  It must be that $c <_{\+L} a$ and $c \in A_1 \subset A_2$.  So $\langle (b, 0), (a, 0), (c, 0)\rangle \in T_b$.  We continue in this fashion, defining a decreasing sequence $c_0 >_{\+L} c_1 >_{\+L} \dots$ satisfying:
\begin{itemize}
\item $c_0 = a$;
\item $c_{n+1}$ is such that $f(\langle (b, 0), (c_0,0), \dots, (c_n, 0)\rangle) = \langle (c_0,k_0), \dots, (c_n, k_n), (c_{n+1}, k_{n+1})\rangle$ for some $k_0, \dots, k_{n+1}$.
\end{itemize}
In this way, $f$ computes the sequence $c_0, c_1, \dots$, which completes the proof.
\end{proof}

We end with some open questions.

\begin{question}
Is there a degree which is high for Scott rank $\ock$ but not high for Scott rank $\ock+1$, and thus not high for isomorphism?
\end{question}

In \cite{cft23}, it was shown that highness for isomorphism is the same as computing a path on every ill-founded computable tree $T \subseteq \omega^{<\omega}$.  This suggests the following question, which seems related to the previous question but is not obviously equivalent.

\begin{question}
Is highness for isomorphism the same as computing a path on every ill-founded computable tree $T \subseteq \omega^{<\omega}$ of Scott rank at most $\ock$?
\end{question}

\printbibliography
\end{document}